\documentclass[11pt]{amsart}
\usepackage{amsmath}
\usepackage{amsfonts}
\usepackage{eucal}
\usepackage{amsthm}
\usepackage{amssymb}
\usepackage[all]{xy}
 \SelectTips{cm}{}
\usepackage[bookmarksnumbered]{hyperref}

\title{Stable left and right Bousfield localisations}
\author{David Barnes}
\address{David Barnes \\ The University of Sheffield \\ School of Mathematics and Statistics \\ Hicks Building \\  Sheffield \\ S3 7RH \\ UK
                }
\email{D.J.Barnes@sheffield.ac.uk}
\thanks{The first author was supported by EPSRC grant EP/H026681/1} 
\author{Constanze Roitzheim}
\address{Constanze Roitzheim \\ University of Glasgow \\ Department of Mathematics \\ 15 University Gardens \\ Glasgow \\ G12 8QW \\ UK
        }
\email{constanze.roitzheim@glasgow.ac.uk}
\thanks{The second author was supported by EPSRC grant EP/G051348/1.}
\date{$24^\text{th}$ April 2012}

\DeclareMathOperator{\id}{Id}
\DeclareMathOperator{\Ho}{Ho}
\DeclareMathOperator{\End}{End}
\DeclareMathOperator{\Hom}{Hom}
\DeclareMathOperator{\Map}{Map}
\DeclareMathOperator{\SSet}{sSet}
\DeclareMathOperator{\sing}{sing}

\DeclareMathOperator{\rightmod}{mod--}

\DeclareMathOperator{\cell}{Cell}

\DeclareMathOperator{\ch}{Ch}

\newcommand{\Sp}{\mathcal{S}}
\newcommand{\MSp}{\mathcal{MS}}
\newcommand{\C}{\mathcal{C}}
\newcommand{\D}{\mathcal{D}}

\newcommand{\co}{\colon \!}
\newcommand{\lradjunction}{\,\,\raisebox{-0.1\height}{$\overrightarrow{\longleftarrow}$}\,\,}
\newcommand{\rladjunction}{\,\,\raisebox{-0.1\height}{$\overleftarrow{\longrightarrow}$}\,\,}

\newcommand{\smashprod}{\wedge}

\newcommand{\sphspec}{{ \mathbb{S} }}

\newtheorem{theorem}{Theorem}[section]
\newtheorem{proposition}[theorem]{Proposition}
\newtheorem{corollary}[theorem]{Corollary}
\newtheorem{lemma}[theorem]{Lemma}

\newtheorem{definition}[theorem]{Definition}
\newtheorem{ex}[theorem]{Example}

\newtheorem{rmk}[theorem]{Remark}

\begin{document}

\begin{abstract}
\noindent 
We study left and right Bousfield localisations of stable
model categories which preserve stability. This follows the lead of
the two key examples:
localisations of spectra with respect to a homology theory and  $A$-torsion
modules over a ring $R$ with $A$ a perfect $R$-algebra. We exploit stability
to see that the
resulting model structures are technically far better behaved than the
general case.
We can give explicit sets of generating cofibrations, show that these
localisations preserve properness and give a complete characterisation
of when they preserve monoidal structures. We apply these results to
obtain convenient assumptions under which a stable model category is
spectral. We then use Morita theory to gain an insight into the nature of right localisation and its homotopy category. We finish with a  correspondence between left and
right localisation.

\end{abstract}

\maketitle 

\section*{Introduction}
Localisations of homotopy theories are one of the most useful techniques in
the tool-kit of an algebraic topologist. Bousfield introduced this concept
by studying topological spaces up to $E_*$-equivalence for $E_*$ a homology theory. This became known as \emph{left Bousfield localisation}. 
Later, the dual concept known as \emph{cellularisation}, or \emph{right Bousfield localisation}, was developed. 
Dwyer and Greenlees discuss this in the context of an algebraic case in \cite{DwyGre02}, where they consider $A$-torsion modules in the case of $A$ a perfect algebra over the ring $R$.

As these two notions of localisation and cellularisation were studied, it became clear that they were best phrased in the language of model categories. 
It was therefore natural to ask if localisation or cellularisation can be performed in a general model category.
A good answer to this was given by Hirschhorn in the book \cite{Hir03}, which discusses general existence questions as well as studying technical properties of left and right localisations.
Left localisation and right localisation are dual notions, 
but the main results of the book are not dual. 
This creates some very interesting differences in the behaviour of
left and right localisations. 

In this paper we restrict our view to stable model categories. Both main examples of $E_*$-localisation of spectra and $A$-torsion $R$-modules are cases of localisations of stable model categories where the result is still stable. However, in general stability is a condition not always preserved by left or right localisation. We begin our work by finding simple necessary and sufficient conditions for the preservation of stability, thus obtaining the notion of \emph{stable left and right localisations} of stable model categories.
In this setting we show that left and 
right localisations are much more tractable than in the general case. 
We are able to see that a right localisation is still cellular
and obtain small and explicit sets of generating cofibrations
and generating acyclic cofibrations
for both left and right localisations. 
Furthermore we see that properness is preserved by left and right localisations.
This is a considerable advantage over the non-stable case as generally, such convenient sets of generating cofibrations and acyclic cofibrations cannot be found. Furthermore, our approach can be viewed as an improvement on the existence results of left and right localisation as the stable case requires fewer technical assumptions on the original model category.

We then exploit our new description of the generating sets
to see that monoidal structures interact very well with localisations
of stable model categories. We then return to the motivating example of
spectra and see that left localisations behave extremely well and 
are easily made stable and monoidal. We also obtain an even simpler set of generating cofibrations and acyclic cofibrations. Dually, we can use our tools to deduce that the category of $A$-torsion modules is a monoidal model category.

One further interesting application of the stable setting is that we are now able to prove that any stable, proper and cellular model category is Quillen equivalent to a spectral model category. Since we now know that stable left localisation preserves properness we are able to combine existing results to obtain a sleeker and more tractable answer than previous results along these lines.

We continue by using Morita theory to show that for a set of homotopically compact objects $K$, right localisation with respect to $K$ is Quillen equivalent to modules over the endomorphism ringoid spectrum of $K$. This shows that the $K$-colocal homotopy category $\Ho(R_K\C)$ of $\C$ is the smallest localising subcategory of $\Ho(\C)$ containing $K$ and provides an explicit description of cellularisation. 

We further show that for any left localisation there is a corresponding right localisation governing the acyclics of this left localisation and vice versa. This allows us to restate the Telescope Conjecture in chromatic homotopy theory in terms of right localisations.

\bigskip
Our results regarding properness, existence and monoidality of left and right localisations as well as their applications show that stable localisations of stable model categories have vast advantages over the general case. Furthermore we have shown that right localisations are not to be dreaded and 
hope that our work will encourage others to use this powerful technique. 

\bigskip

This paper is organised as follows. In Section \ref{sec:examples} we establish the notions of left and right Bousfield localisations of model categories. We then discuss some standard examples, namely localisation of spectra with respect to a homology theory and $A$-torsion $R$-modules where $A$ is a perfect algebra over the commutative ring $R$.

In Section \ref{sec:modelcats} we recall some definitions in the context of model categories, namely stability, framings, properness and cofibrant generation. These technical definitions will play a crucial role to our work.

Section \ref{sec:stableleftloc} contains the first key results concerning left Bousfield localisation $L_S\C$. We define what it means for a set of maps $S$ to be stable and then show that under the assumption of stability of $S$, localisation preserves stability and properness. Further, we give a simple set of generating cofibrations and acyclic cofibrations for $L_S\C$.
Section \ref{sec:stablerightloc} deals with analogous results for the dual case of right Bousfield localisation $R_K\C$, where $K$ is a set of objects of $\C$. 

The following pair of sections, \ref{sec:monoidalleft} and \ref{sec:monoidalright}, examine the interaction of  left and right localisations with monoidal structures. More specifically, for a monoidal model 
category $\C$ we give necessary and sufficient conditions on $S$ and $K$ so that $L_S\C$ and $R_K\C$ are again monoidal model categories and prove some universal properties. We also apply our results to the leading examples of spectra and $A$-torsion $R$-modules. 

Section \ref{sec:spectral} uses the fact that stable left localisations preserve properness to obtain convenient conditions under which a stable model category is Quillen equivalent to a spectral one.

In Section \ref{sec:morita}, we use the Morita theory of Schwede and Shipley to gain further insight into right localisations when the object set $K$ consists of homotopically compact objects. 
In particular we are able to generalise the results of \cite{DwyGre02}
and obtain conditions under which a right localisation at a set of objects $K$ 
is Quillen equivalent to a left localisation at a set of maps $S$. 

Finally, in section \ref{sec:correspondence} we update the important correspondence between 
cellularisations and acyclicisations to the language of left and right localisation 
by comparing colocal objects to acyclic objects, leading to an alternative description of the Telescope Conjecture.

\bigskip
We would like to thank Denis-Charles Cisinski and John Greenlees for motivating conversations.

\section{Examples of left and right Bousfield localisation}\label{sec:examples}

Let $E_*$ be a generalised homology theory. In the 1970s Bousfield considered the resulting homotopy categories of spaces and spectra after inverting $E_*$-isomorphisms rather than $\pi_*$-isomorphisms. Those homotopy categories are especially sensitive with respect to phenomena related to $E_*$.  To talk about these constructions in a set-theoretically rigid manner, they were increasingly placed in a model category context in the subsequent decades. We are going to recall some definitions and results in this section.

\begin{definition}
A map $f \co X \to Y$ of simplicial sets or spectra is an \textbf{$E$-equivalence} if $E_*(f)$ is an isomorphism.
A simplicial set or a spectrum $Z$ is \textbf{$E$-local} if $$f^* \co [Y,Z] \to [X,Z]$$ is an isomorphism for
all $E$-equivalences $f \co X \to Y$. A simplicial set of spectrum $A$ is \textbf{$E$-acyclic} if
$[A,Z]$ consists of only the trivial map, for all $E$-acyclic $Z$.
An $E$-equivalence from $X$ to an $E$-local object $Z$ is called an \textbf{$E$-localisation}.
\end{definition}

This then gives rise to the following, see \cite{Bou75} and \cite{Bou79}.

\begin{theorem}\label{thm:eloc}
Let $E$ be a homology theory and $\C$ be the category of simplicial sets
or spectra. Then there is a model structure $L_E\C$ on $\C$ such that
\begin{itemize}
\item the weak equivalences are the $E_*$-isomorphisms,
\item the cofibrations are the cofibrations $\C$,
\item the fibrations are those maps with the right lifting property with respect to 
cofibrations that are also $E_*$-isomorphisms.
\end{itemize}
\end{theorem}

This can be seen as a special case of a more general result by Hirschhorn.
For $X, Y \in \C$, we let $\Map_\C(X,Y)$ denote the homotopy function object, 
which is a simplicial set, see \cite[Chapter 17]{Hir03} and Section \ref{sec:modelcats}.

\begin{definition}\label{def:local}
Let $S$ be a set of maps in $\C$. Then an object $Z \in \C$ is \textbf{$S$-local} if
\[
\Map_\C(s,Z): \Map_\C(B,Z) \longrightarrow \Map_\C(A,Z)
\]
is a weak equivalence in simplicial sets for any
$s: A \longrightarrow B$ in $S$.
A map $f:X \longrightarrow Y \in \C$ is an \textbf{$S$-equivalence} if
\[
\Map_\C(f,Z): \Map_\C(Y,Z) \longrightarrow \Map_\C(X,Z)
\]
is a weak equivalence for any $S$-local $Z \in \C$.
An object $W \in \C$ is \textbf{$S$-acyclic} if
$$
\Map_\C(W,Z) \simeq *
$$
for all $S$-local $Z \in \C$. 
\end{definition}

A \textbf{left Bousfield localisation} of a model category $\C$ with respect to a class of maps $S$ is a new model structure $L_S\C$ on $\C$ such that
\begin{itemize}
\item the weak equivalences of $L_S\C$ are the $S$-equivalences, 
\item the cofibrations of $L_S \C$ are the cofibrations of $\C$,
\item the fibrations of $L_S\C$ are those maps that have the right lifting property with respect to cofibrations that are also $S$-equivalences. 
\end{itemize}

Hirschhorn proves that with some minor assumptions on $\C$, $L_S\C$ exists if $S$ is a set. In the case of homological localisation as in Theorem \ref{thm:eloc} the class $S$ is initially the class of $E_*$-isomorphisms, which is not a set. Hence, the key to proving the existence of homological localisations is to show that there is a \emph{set} $S$ whose $S$-equivalences are exactly the $E_*$-isomorphisms. 

For example, this has been done for spectra, specifically for $\mathbb{S}$-modules in the sense of \cite{EKMM}. In Section VIII.1 they show that there is a set $\mathcal{J}_E$ of generating $E$-acyclic cofibrations, 
that is, a morphism of spectra is an $E$-acyclic fibration if and only if it has the right lifting property with respect to all elements of $\mathcal{J}_E$. This implies that $L_E = L_{\mathcal{J}_E}$ as both localisations then possess the same fibrant objects and in particular the same local objects. 
Similar results exist for symmetric spectra, sequential spectra and orthogonal spectra and their 
equivariant counterparts. 

\bigskip

We are now going to turn to a dual notion, namely right Bousfield localisation.
We should be talking about right localisation with respect to a class of maps, 
however the existence theorem for right localisations, \cite[Theorem 5.1.1]{Hir03},
is not dual. That theorem, which we give in Section \ref{sec:stablerightloc}, starts with a set of objects 
rather than a set of maps. Thus we always word right localisations
in terms of a  set (or class) of objects. 

\begin{definition}
Let $\C$ be a model category and $K$ a class of objects of $\C$.  We say that a map $f: A \longrightarrow B$ of $\C$ is a \textbf{$K$-coequivalence} if
\[
\Map_\C(X,f): \Map_\C(X,A) \longrightarrow \Map_\C(X,B)
\]
is a weak equivalence of simplicial sets for each $X \in K$. An object $Z \in \C$ is \textbf{$K$-colocal} if
\[
\Map_\C(Z,f): \Map_\C(Z,A) \longrightarrow \Map_\C(Z,B)
\]
is a weak equivalence for any $K$-coequivalence $f$. An object $A\in\C$ is \textbf{$K$-coacyclic} if $\Map_\C(W,A) \simeq *$ for any $K$-colocal $W$.
\end{definition}

There are many other similar names for these terms, in particular 
\cite[Definition 5.1.3]{Hir03} uses the term $K$-colocal 
equivalences for $K$-coequivalences.

A \textbf{right Bousfield localisation} of $\C$ with respect to $K$ is a model structure $R_K\C$ on $\C$ such that
\begin{itemize}
\item the weak equivalences are $K$-coequivalences
\item the fibrations in $R_K\C$ are the fibrations in $\C$
\item the cofibrations in $R_K\C$ are those morphisms that have the left lifting property with respect to fibrations that are $K$-coequivalences.
\end{itemize}

When $K$ is a set rather than an arbitrary class, Hirschhorn showed in \cite[Theorem 5.1.1]{Hir03} that, under some assumptions on $\C$, $R_K \C$ exists. This is going to be discussed in more detail later in 
Section \ref{sec:stablerightloc}. 

\bigskip
An example of right Bousfield localisation of modules over a ring $R$ was discussed by Dwyer and Greenlees in \cite{DwyGre02}. A \textbf{perfect} $R$-module $A$ is isomorphic to a differential graded $R$-module of finite length which is finitely generated projective in every degree. This is equivalent to $A$ being small, meaning that $R\Hom_R(A,-)$, the derived functor of $\Hom_R(A,-)$, commutes with arbitrary coproducts.

Dwyer and Greenlees consider right localisation of the category of $R$-modules with respect to $K = \{ A \}$ where $A$ is perfect. In their paper, they call the thus arising $\{A\}$-coequivalences ``$E$-equivalences'', referring to the functor $E(-)=R\Hom_R(A,-)$. The $\{A\}$-colocal objects are referred to as ``$A$-torsion modules''. For example, in the case of $R=\mathbb{Z}$ and $A=( \mathbb{Z} \xrightarrow{\cdot p} \mathbb{Z}) \simeq \mathbb{Z}/p$, an $R$-module $X$ is $p$-torsion if and only if it has $p$-primary torsion homology groups.

\bigskip
In \cite{DwyGre02}, Dwyer and Greenlees also compare this version of right localisation with a dual notion of left localisation. In the same set-up they consider left localisation with respect to the class $S$ of $R\Hom_R(A,-)$-isomorphisms. They call the resulting $S$-local $R$-modules ``$A$-complete''. In their Theorem 2.1 they show that the derived categories of $A$-torsion and $A$-complete modules are equivalent. 
We will provide a generalisation of this type of result in Section \ref{sec:morita}.

\section{Some model category techniques}\label{sec:modelcats}

We will recall some technical facts about stable model categories. The homotopy category of any pointed model category can be equipped with an adjoint functor pair 
\[
\Sigma:  \Ho(\C) \lradjunction \Ho(\C) : \Omega
\]
where $\Sigma$ is called the \textbf{suspension functor} and $\Omega$ the \textbf{loop functor}. 
Let $X \in \Ho(\C)$ be fibrant and cofibrant in $\C$. We factor the map
\[
X \longrightarrow *
\]
into a cofibration and a weak equivalence
\[
X \rightarrowtail X \stackrel{\sim}{\longrightarrow} *.
\]
The suspension $\Sigma X$ of $X$ is defined as the pushout of the diagram
\[
CX \leftarrowtail X \rightarrowtail CX.
\]
Dually the loops on $X$ are defined as the pullback of
\[
PX \twoheadrightarrow X \twoheadleftarrow PX
\]
where 
\[
* \stackrel{\sim}{\longrightarrow} PX \twoheadrightarrow X
\]
is a factorisation of $* \longrightarrow X$ into a weak equivalence and a fibration. For example, in the case of topological spaces this gives the usual loop and suspension functors. For chain complexes 
of $R$-modules, denoted $\ch(R)$, the suspension and loop functors are degree shifts of chain complexes.

\begin{definition}
A model category $\C$ is \textbf{stable} if $\Sigma$ and $\Omega$ are inverse equivalences of categories.
\end{definition}

Thus, topological spaces are not stable whereas $\ch(R)$ is. 

\bigskip

An alternative description of $\Sigma$ and $\Omega$ uses the technique of \textbf{framings} which is a generalisation of the notion of a simplicial model category. Recall that a simplicial model category is a model category that is enriched, tensored and cotensored over the model category of simplicial sets satisfying some adjunction properties \cite[Definition II.2.1]{GJ98}. Further, these functors are supposed to be compatible with the respective model structures on the model category $\C$ and simplicial sets $\SSet_*$. Not every model category can be given the structure of a simplicial model category, but framings at least give a similar structure up to homotopy.
For details, see e.g. \cite[Chapter 5]{Hov99}, \cite[Chapter 16]{Hir03} or \cite[Section 3]{BarRoi11b}. 

Let $\C$ be a pointed model category and $A \in \C$ a fixed object. Framings give adjoint Quillen functor pairs
\[
\begin{array}{rrl}
A \otimes (-)  : & \SSet_* \lradjunction \C & : \Map_l(A,-) \\
A^{(-)} : & \SSet^{op}_* \rladjunction \C & : \Map_r(-,A).
\end{array}
\]
Unfortunately the construction is not rigid enough to equip any model category with the structure of a simplicial model category. The reason for this is that for two fixed objects $A$ and $B$ the above defined ``left mapping space' $\Map_l(A,B)$ and ``right mapping space'' $\Map_r(A,B)$ only agree up to a zig-zag of weak equivalences. However, the above functors possess total derived functors, giving rise to an adjunction of two variables
\[
\begin{array}{r@{\ : \ }ll}
- \otimes^L - & \Ho(\C) \times \Ho(\SSet) \longrightarrow \Ho(\C),  \\
R\Map(-,-) & \Ho(\C)^{op} \times \Ho(\C) \longrightarrow  \Ho(\SSet_*), \\
R(-)^{(-)} &  \Ho(\SSet)^{op} \times \Ho(\C) \longrightarrow \Ho(\C).
\end{array}
\]

\begin{theorem}[Hovey]
Let $\C$ be a pointed model category. Then its homotopy category $\Ho(\C)$ is a $\Ho(\SSet_*)$-module category.
\end{theorem}

In particular the homotopy function complex $\Map_\C$ is weakly equivalent to 
$R \Map$. Hence we will abuse notation and only write $\Map$ instead of $R\Map$ or $\Map_\C$.
The suspension and loop functors can also be described using framings, see \cite[Chapter 6]{Hov99}. 

\begin{lemma}
Let $\mathbb{S}^1 \in SSet_*$ denote the simplicial circle. Then
\[
\Sigma X \cong X \otimes^L \mathbb{S}^1 \,\,\mbox{and}\,\,\, \Omega X \cong (RX)^{\mathbb{S}^1}.
\] \qed
\end{lemma}

\bigskip
Another model category notion relevant to this paper is properness. This definitions does not seem important at first sight but is crucial to many of the results about the existence of a localisation.

\begin{definition}A model category is \textbf{left proper} if every pushout of a weak equivalence along a cofibration is again a weak equivalence. Dually, it is said to be \textbf{right proper} if every pullback of a weak equivalence along a fibration is again a weak equivalence. It is \textbf{proper} if it is both left and right proper.
\end{definition}

\bigskip
Recall that a model category $\C$ is said to be cofibrantly generated if there are \emph{sets} of maps (rather than classes) that generate the cofibrations and acyclic cofibrations of $\C$. More precisely,

\begin{definition}
A model category $\C$ is \textbf{cofibrantly generated} if there exist sets of maps $I$ and $J$ such that
\begin{itemize}
\item a morphism in $\C$ is a fibration if and only if it has the right lifting property with respect to all elements in $I$,
\item a morphism in $\C$ is an acyclic fibration if and only if it has the right lifting property with respect to all elements in $J$.
\end{itemize}
Further, $I$ and $J$ have to satisfy the \textbf{small object argument}, that is, the domains of the elements of $I$ (and $J$) are small relative to $I$ (respectively $J$).
\end{definition}

For details of smallness and the small object argument see \cite[Section 10.5.14]{Hir03}. The concept of cofibrant generation is crucial to some statements about model categories and in general allows many proofs to be greatly simplified. 

A \textbf{cellular model category} is a cofibrantly generated model category where the generating cofibrations and acyclic cofibrations satisfy some more restrictive properties regarding smallness, see \cite[Definition 12.1.1]{Hir03}. Not every cofibrantly generated model category is cellular, but many naturally occurring model categories are. Examples include simplicial sets, topological spaces, chain complexes of $R$-modules, sequential spectra, symmetric spectra, orthogonal spectra and EKMM $\mathbb{S}$-modules.

\section{Stable left localisation}\label{sec:stableleftloc}

In this section we introduce the notion of left Bousfield localisation with respect to a {``stable''} class of morphisms. We then show that in this framework, the left Bousfield localisation of a stable model category remains stable. We will see that if $\C$ is a stable model category and $S$ is a stable class of maps, then $L_S\C$ (provided it exists) is right proper whenever $\C$ is. Furthermore, if $\C$ is cellular and proper, we can specify a very convenient set of generating cofibrations and acyclic cofibrations for $L_S \C$.

\bigskip
In Section \ref{sec:examples} we defined the notion of $S$-local objects and $S$-equivalences for a class of maps $S \subset \C$.
Note that elements $s \in S$ are automatically $S$-equivalences, although the converse does not have to be true. For example, any weak equivalence in $\C$ is an $S$-equivalence.

By \cite[Theorem 4.1.1]{Hir03}, in nice cases the $S$-local model structure on $\C$ exists. In particular, this result requires $S$ to be a set. 

\begin{theorem}[Hirschhorn]
Let $\C$ be a left proper, cellular model category. Let $S$ be a set of maps in $\C$. Then there is a model structure $L_S\C$ on the underlying category $\C$ such that
\begin{itemize}
\item weak equivalences in $L_S \C$ are $S$-equivalences,
\item the cofibrations in $L_S \C$ are the cofibrations in $\C$.
\end{itemize}
The fibrations in this model structure are called $S$-fibrations.
\end{theorem}

Note that fibrant replacement $U_S$ in $L_S \C$ is a localisation, that is, an $S$-equivalence
\[
X \longrightarrow U_S(X)
\]
where $U_S(X)$ is $S$-local. It is important to distinguish between fibrant in $\C$, $S$-fibrant and $S$-local. The first two are model category conditions, the third is a condition on the homotopy type of an object. Note that an object is $S$-fibrant if and only if it is $S$-local and fibrant in $\C$. 

The functors $\Sigma$ and $\Omega$ interact well with 
homotopy function complexes since all three can be defined
via framings. In particular we the weak equivalences below. 
\[
\Map_\C(\Sigma X, Y) \simeq \Map_\C(X, \Omega Y) \simeq \Omega\Map(X,Y)
\]
Combining this adjunction with Definition \ref{def:local} we obtain the following pair of facts.
\begin{itemize}
\item The class of $S$-equivalences is closed under $\Sigma$.
\item The class of $S$-local objects is closed under $\Omega$.
\end{itemize}

\begin{definition}
Let $S$ be a class of maps in $\C$. We say that $S$ is \textbf{stable} if the collection of $S$-local objects is closed under $\Sigma$.
\end{definition}

\begin{ex}
Let $\C$ be either the category of pointed simplicial sets or the category of spectra. 
Let $S$ be the class of $E_*$-isomorphisms for a generalised homology theory $E$. 
Then $S$ is not a set in either of these two cases, but it is stable and $L_S \C$ exists.
\end{ex}

A simple adjunction argument shows the following.

\begin{lemma}\label{lem:stableconditions}
If $\C$ is a stable model category, 
then a class of maps $S$ is stable if and only if
the collection of $S$-equivalences is closed under $\Omega$.
In particular, if the set $S$ is closed under $\Omega$, then the set $S$ is stable. \qed
\end{lemma}

\begin{rmk}\label{rmk:usehomotopyclasses}
The definitions of $S$-equivalences and $S$-local objects are 
given in terms of homotopy function complexes, denoted $\Map(-,-)$.
However since we work in a stable context we can rewrite these definitions into 
more familiar forms involving $[-,-]_*^{\C}$, the graded set of maps in the
homotopy category of $\C$. 

By \cite[Theorem 17.7.2]{Hir03}, there is a natural isomorphism
$$\pi_0 \Map_\C(X,Y) \cong [X,Y]^{\C}.$$ 
It follows that $$\pi_n \Map_\C(X,Y) \cong [X,Y]^{\C}_n \,\,\,\mbox{for}\,\,\, n \geqslant 0.$$
Similarly, $$\pi_n \Map_\C(\Omega^{k} X,Y) \cong [X,Y]^{\C}_{n-k} \,\,\,\mbox{for}\,\,\, n,k \geqslant 0.$$

It follows that $f \co X \to Y$ is an $S$-equivalence if and only if
the map 
\[
[f,Z]_*^{\C} \co [Y,Z]_*^{\C} \longrightarrow [X,Z]_*^{\C}
\] 
is an isomorphism of graded abelian groups. 
\end{rmk}

\begin{proposition}\label{prop:leftstable}
Let $\C$ be a stable model category, let $S$ be a class of maps and
assume that $L_S\C$ exists.
Then $L_S\C$ is a stable model category if and only if
$S$ is a stable class of maps.
\end{proposition}

\begin{proof}
The homotopy category of $L_S\C$ is equivalent to the full subcategory of
$\Ho(\C)$ with object class given by the $S$-local objects.
In Section \ref{sec:modelcats} we defined the functor $\Omega$ in
terms of framings.
In particular the restriction of the functor $$\Omega \co \Ho(\C) \to \Ho(C)$$
to $\Ho(L_S \C)$ is naturally isomorphic to the the desuspension functor
on $\Ho(L_S \C)$ coming from framings on the model category $L_S \C$.
We thus see that $$\Omega \co \Ho(L_S \C) \to \Ho(L_S \C)$$
is a fully faithful functor as it is the restriction of an
equivalence to a full subcategory.
We must show that it is essentially surjective.
Consider some $S$-local $X$, then the suspension $\Sigma X$ of
$X$ is also $S$-local as $S$ is a stable
class of maps.
Hence $\Sigma X$ is in $\Ho(L_S \C)$ and the unit of the adjunction
$(\Sigma, \Omega)$ on $\Ho(\C)$ gives an isomorphism
$$X \to \Omega \Sigma X$$ in $\Ho(\C)$ and hence in $\Ho(L_S \C)$.

\medskip
For the converse, assume that $L_S \C$ is stable,
and consider some $S$-local $X$.
Then $$\Omega \co  \Ho(L_S \C) \to \Ho(L_S \C)$$
is an essentially surjective functor. Hence there is some
$S$-local $Y$ such that $\Omega Y$ is isomorphic to $X$ in
$\Ho(L_S \C)$. It follows that $\Omega Y$ is isomorphic
to $X$ in $\Ho(\C)$. Then by stability of $\C$,
$\Sigma \Omega Y \cong Y$ is isomorphic to
$\Sigma X$ in $\Ho(\C)$. Since $Y$ is $S$-local,
it follows that $\Sigma X$ must also be $S$-local,
hence $S$ is a stable class of maps.
\end{proof}

So for a stable class $S$, the homotopy category of $\Ho(L_S\C)$ is
triangulated, which is the vital ingredient of the next proposition.
By \cite[Proposition 3.4.4]{Hir03} we know that $L_S\C$ is left proper
if $\C$ is left proper. But we now also have the following.

\begin{proposition}\label{prop:rightproper}
Let $\C$ be a stable, right proper model category and $S$ a stable class of maps. If $L_S\C$ exists, then it is right proper.
\end{proposition}

\begin{proof} We consider the following pullback square
\[
\xymatrix{ 
X' \ar[d]_{u} \ar[r]^{p'} & Y' \ar[d]^{v} \\
X \ar[r]_{p} & Y
}
\]
where $p$ is an $S$-fibration (and hence a fibration in $\C$) and $v$ is an $S$-equivalence. Our goal is to show that $u$ is also an $S$-equivalence.

The fibre of a map $p: X \longrightarrow Y$ is defined as the pullback of the diagram
\[
X \stackrel{p}{\longrightarrow} Y \longleftarrow *.
\]
Since $\C$ is right proper, \cite[Proposition 13.4.6]{Hir03} 
tells us that the fibre of $p$ is also the homotopy fibre of $p$, $Fp$. 
Similarly the fibre of $p'$ is also its homotopy fibre $Fp'$. 
The fibres are isomorphic since we started with a pullback square, 
hence the homotopy fibres are weakly equivalent. 
Now consider the comparison of exact triangles in $\Ho(\C)$
\[
\xymatrix{\Omega Y' \ar[r]\ar[d]_{\Omega v} & Fp' \ar[r]\ar[d]_{\cong} & X' \ar[r]\ar[d]_{u} & Y' \ar[d]_{v} \\
\Omega Y \ar[r] & Fp \ar[r] & X \ar[r] & Y.
}
\]
Since $S$ is stable, this is also a morphism of exact triangles in $\Ho(L_S\C)$. 
Furthermore, $\Omega v$ is an $S$-equivalence. 
Hence the five lemma for triangulated categories implies that $u$ is also an $S$-equivalence, which is what we wanted to show.
\end{proof}

We now need a pair of technical lemmas, the second of which gives a useful characterisation of 
$S$-fibrations.

\begin{lemma}\label{lem:triangle} Let $\C$ be a stable model category and $S$ a stable class of maps.
Assume that $L_S \C$ exists and that we have a commutative triangle in $\C$
\[
\xymatrix{ X \ar[rr]^{u} \ar[dr]_{p} & & Y \ar[dl]^{q} \\
 & B &
 }
\]
such that the homotopy fibres of $p$ and $q$ are $S$-local. Then $u$ is an $S$-equivalence if and only if it is a weak equivalence in $\C$. 
\end{lemma}

\begin{proof} 
The above gives a distinguished triangle in $\Ho(\C)$ and hence in $\Ho(L_S \C)$
\[
\xymatrix{ \Omega B \ar[r] \ar@{=}[d] & Fp \ar[r] \ar[d]_{v} & X \ar[r]^{p} \ar[d]_{u} & B  \ar@{=}[d] \\
\Omega B \ar[r] & Fq \ar[r] & Y \ar[r]^{q} & B
}
\]
Since $S$-equivalences between $S$-local objects are weak equivalences, the result follows. 
\end{proof}

\begin{lemma}\label{lem:fibrant}
 Let $\C$ be a stable right proper model category such that $L_S \C$ exists. Consider a fibration $p: X \longrightarrow Y$ in $\C$. Then $p$ is an $S$-fibration if and only if the fibre of $p$ is $S$-fibrant.
\end{lemma}

\begin{proof}
Since pullbacks of fibrations are fibrations, the fibre of an $S$-fibration is $S$-fibrant. 
Conversely, assume that the fibre $Fp$ is $S$-fibrant. Since $\C$ is assumed to be right proper, $Fp$ is also the homotopy fibre of $p$. We factor $p$ in $L_S\C$ as below.
\[
\xymatrix{ X \ar[dr]_{p} \ar@{>->}[rr]^{j}_{\sim} & & B \ar@{->>}[dl]^{q} \\
& Y &
}
\]
Since the homotopy fibres of $p$ and $q$ are both $S$-fibrant and hence $S$-local, $j$ is a weak equivalence in $\C$ by Lemma \ref{lem:triangle}. As $p$ is a fibration in $\C$, it has the right lifting property with respect to $j$
\[
\xymatrix{ X \ar@{=}[r]\ar@{>->}[d]_{\sim}^{j} & X \ar@{>>}[d]^{p} \\
B \ar[r]_{q} \ar@{.>}[ur]_{f} & Y.
}
\]
The commutative diagram
\[
\xymatrix{ X \ar[r]^{j}\ar[d]_{p}& B \ar[r]^{f}\ar[d]_{q}& X \ar[d]_{p}\\
Y \ar@{=}[r] & Y \ar@{=}[r] & Y
}
\]
shows that $p$ is a retract of the $S$-fibration $q$ and hence an $S$-fibration itself, which is what we wanted to show.
\end{proof}

We are now almost ready to prove our main theorem for this section which gives a very convenient description of the generating cofibrations and acyclic cofibrations of $L_S\C$ when $S$ is assumed to be stable. For technical reasons we want 
$S$ to consist of cofibrations between cofibrant objects. 
Any map is weakly equivalent to such a map and 
changing the maps in $S$ up to weak equivalence 
does not alter the weak equivalences of $L_S \C$, so this is no restriction. 

Before we give the theorem, we need an extra piece of terminology, see
\cite[Definition 3.3.8]{Hir03}. Recall that in Section \ref{sec:modelcats} we defined the action of simplicial sets on $\C$ via framings, which gives a bifunctor
\[
- \otimes - : \C \times \SSet_* \longrightarrow \C.
\]
In particular, if the model category $\C$ is simplicial, then this agrees with the given simplicial action on $\C$. 

\begin{definition}
Let $f \co A \to B$ be a map of $\C$ and let $i_n \co \partial \Delta[n]_+ \to \Delta[n]_+$
be the standard inclusion of pointed simplicial sets. 
Then we define a set of \textbf{horns} on a set of maps $S$ in $\C$ to be 
the set of maps of $\C$ below. 
\[
\Lambda S = 
\Bigl\{
{f} \square i_n \co {A} \otimes \Delta[n]_+ 
\underset{{A} \otimes \partial \Delta[n]_+}{\amalg}
{B} \otimes \partial \Delta[n]_+
\to 
{B} \otimes \Delta[n]_+
| \ (f \co A \to B) \in S, \ n \geqslant 0
\Bigr\}
\]
\end{definition}
In the above definition, one has to choose cosimplicial resolutions of $A$ and $B$ 
such that $f$ induces a Reedy cofibration between the resolutions. 
However the model structure $L_S \C$ is independent of these choices. 
Note that if $S$ consists of cofibrations between cofibrant objects, so does 
$\Lambda S$. 

\begin{theorem}\label{thm:stableleftlocal}
Let $\C$ be a stable, proper, cellular model category with generating cofibrations $I$ and generating acyclic cofibrations $J$. 
Let $S$ be a stable set of cofibrations between cofibrant objects. Then $L_S\C$ is cellular, with the generating cofibrations of $L_S\C$ given by $I$ and the generating acyclic cofibrations are given by $J \cup \Lambda S$.
\end{theorem}

\begin{proof} Note that our assumptions imply that $L_S\C$ exists and is cellular.
Our task is to show that $J \cup \Lambda S$ is a set of generating acyclic cofibrations.
First of all, let us prove the following claim. Assume that $T$ is a set of cofibrations that are also $S$-equivalences Further, assume that $Z \in \C$ is $S$-fibrant if and only if the map $Z \longrightarrow *$ has the right lifting property with respect to $J \cup T$. Then a map $f$ is an $S$-fibration if and only if it has the right lifting property with respect to $J \cup T$.

\medskip
If $f$ is an $S$-fibration, then of course it has the right lifting property with respect to both $J$ and $T$. So let us assume conversely that $f: X \longrightarrow Y$ is a map that has the right lifting property with respect to $J \cup T$. We want to use Lemma \ref{lem:fibrant} and show that $F$, the fibre of $f$, is $S$-fibrant. 
Take some $j \co A \to B$ in $J \cup T$ and consider a lifting square between $j$ and $F \to \ast$. 
We may extend that square to include $f$, as below.
\[
\xymatrix{ A \ar[r]\ar[d]^{j} & F \ar[r]\ar[d] & X\ar[d]_{f} \\
B \ar[r] \ar@{.>}[urr] & \ast \ar[r] & Y
}
\]
Since $f$ is assumed to have the right lifting property with respect to $j$, the lift in the diagram exists. By the universal property of the pullback, there is also a map $B \longrightarrow F$ making the left square commute. Thus, $F$ also has the right lifting property with respect to $J \cup T$.
Hence by our assumptions and Lemma \ref{lem:fibrant}, $f$ is an $S$-fibration. 

\medskip
Now that we have proven our claim, we are ready to prove that $J \cup \Lambda S$ is indeed a set of 
generating acyclic cofibrations of $L_S\C$, we have to show that the assumptions of the 
above claim are verified. This means we have to show that an object $Z$ is $S$-fibrant 
if and only if the map $Z \longrightarrow *$ has the right lifting property with respect to $J \cup \Lambda S$.

The maps of $J \cup \Lambda S$ are cofibrations that are $S$-local equivalences 
by \cite[Proposition 4.2.3]{Hir03}. Hence if $Z$ is $S$-fibrant, then
$Z \longrightarrow *$ has the right lifting property with respect to $J \cup \Lambda S$.
For the converse, we use \cite[Proposition 4.2.4]{Hir03}, noting that our naming conventions
are slightly different to the reference. 

To prove that this model category is cellular we must check the conditions of \cite[Definition 12.1.1]{Hir03}. This amounts to proving that the domains of $S$ are small relative to $I$. The domains of $S$ are cofibrant, hence they are small with respect to $I$ by \cite[Lemma 12.4.2]{Hir03}.
\end{proof}

\begin{rmk}\label{rmk:noncellular}
The work of \cite{Hir03} uses in an essential manner the assumption that $\C$ is cellular
to obtain a set of generating set of acyclic cofibrations for $L_S \C$. The reference then uses this set 
to show that $L_S \C$ exists. We have used stability to find such a set 
and then used the assumption that $\C$ is cellular
to see that this set satisfies the conditions of the small object argument.

Hence we have a partial refinement of the above theorem to the case when $\C$ is not cellular. 
Assume that $\C$ is a stable, proper cofibrantly generated model category and $S$ is a stable set
of cofibrations. If the domains of $J \cup \Lambda S$ are small relative to 
the class of transfinite compositions of pushouts of $J \cup \Lambda S$, 
then $L_S \C$ exists and is cofibrantly generated by the sets 
$I$ and $J \cup \Lambda S$. Furthermore it is stable and proper. 
\end{rmk}

Theorem \ref{thm:stableleftlocal} is a considerable improvement on the general situation where $\C$ has not been assumed to be stable. 
Without stability, the results of \cite{Hir03} only prove the existence of some set
of generating acyclic cofibrations. Indeed, the set $J \cup \Lambda S$ is not always a generating
set of acyclic cofibrations for $L_S \C$, as shown by \cite[Example 2.1.6]{Hir03} which we will spell out below.
The proof that $L_S \C$ exists and is cofibrantly generated 
in the unstable case uses the Bousfield-Smith cardinality argument.
So in general it is all but impossible to obtain a nice
description of the generating acyclic cofibrations from the proof. 

\begin{ex}
Consider the model category of topological spaces with weak equivalences the weak homotopy equivalences. 
Let $n >0$ and let $f \co S^{n} \to D^{n+1}$ be the inclusion. We now look at localisation with respect to $S=\{f\}$.

The path space fibration $$p \co P K(\mathbb{Z}, n) \to K(\mathbb{Z}, n)$$ has the right lifting property
with respect to $J \cup \Lambda \{ f \}$. Hence every 
$J \cup \Lambda \{ f \}$-cofibration has the left lifting property with respect to $p$. 
But the cofibration $\ast \to S^n$ does not have this left lifting property. 
The composite map $\ast \to S^n \to D^{n+1}$ is clearly an $\{f\}$-local equivalence
as is $f$ itself. Hence $\ast \to S^n$ is a cofibration and an 
$\{f\}$-local-equivalence that is not a 
$J \cup \Lambda \{ f \}$-cofibration.
\end{ex}

We can also use Theorem \ref{thm:stableleftlocal} to consider smashing localisations of spectra. Recall that Bousfield localisation of a model category of spectra $\Sp$, such as symmetric spectra or EKMM $\mathbb{S}$-modules 
is called \textbf{smashing} if for every spectrum $X$ the map 
\[
\lambda \wedge^L \id_X: X \longrightarrow X \wedge^L L_E \mathbb{S}
\]
is an $E$-localisation. 

\begin{lemma}\label{lem:smashing}
If localisation with respect to $E$ is smashing, then $L_E \Sp = L_\Gamma\Sp$ for
\[
\Gamma = \{ \Sigma^n \lambda: \mathbb{S}^n \longrightarrow L_E \mathbb{S}^n \,\,|\,\,n \in \mathbb{Z} \}.
\]
\end{lemma}

\begin{proof}
Every element in $\Gamma$ is an $E$-equivalence, hence every $\Gamma$-equivalence is an $E$-equivalence. Let us now consider the following commutative diagram, where $f: X \longrightarrow Y$ is a map of spectra and $Z$ is $\Gamma$-local.
\[
\xymatrix@C+1cm{ [Y,Z]_* \ar[r]^{f^*} & [X,Z]_*  \\
[Y \wedge^L L_E \mathbb{S}, Z]_* \ar[r]^{(f \smashprod^L L_E \mathbb{S})^*} \ar[u]^{\cong}& 
[X \wedge^L L_E \mathbb{S}, Z]_* \ar[u]_{\cong}
}
\]
The vertical arrows are isomorphisms because the map 
$X \longrightarrow X \wedge^L L_E \mathbb{S}$ 
is a $\Gamma$-equivalence and $Z$ is $\Gamma$-local. 
To see this, note that the class of objects $X$ for which this is a $\Gamma$-equivalence is closed under coproducts and exact triangles, and contains the sphere.

Now let $f$ be an $E$-equivalence. By assumption this is equivalent to  $f \wedge^L L_E \mathbb{S}$ being a weak equivalence. This implies that the bottom row of the commutative square is an isomorphism. Hence the top row is an isomorphism and thus $f$ is a $\Gamma$-equivalence. 
\end{proof}

\begin{corollary}\label{cor:smashingcofibs}
Let $\Sp$ be the model category of symmetric spectra or EKMM $\mathbb{S}$-modules with generating cofibrations $I$ and acyclic cofibrations $J$. Let $L_E$ be a smashing Bousfield localisation with respect to a homology theory $E$. Then $L_E \Sp$ is proper, stable and cellular with generating cofibrations $I$ and generating acyclic cofibrations $J \cup \Lambda \Gamma$. \qed
\end{corollary}

A further refinement on the generating sets appears as Corollary \ref{cor:spectragenerating}.

\section{Stable right localisations}\label{sec:stablerightloc}

In this section we are going to introduce the notion of right Bousfield localisation with respect to a \emph{stable} class of objects. We then proceed by showing that in this framework, the right Bousfield localisation of a stable model category remains stable. We will see that if $\C$ is a stable model category and $K$ is a stable class of maps, then $R_K \C$ (provided it exists) is left proper whenever $\C$ is. Furthermore, if $\C$ is cellular and right proper, we can specify a very convenient set of generating cofibrations and acyclic cofibrations for $R_K \C$.

Right Bousfield localisation is the dual notion to left Bousfield localisation as we have mentioned above. 
We defined $K$-coequivalences and $K$-colocal objects in Section \ref{sec:examples}.
Note that our definitions imply that any object of $K$ is $K$-colocal, but the converse is not necessarily true. Also, any weak equivalences of $\C$ is a $K$-coequivalence.

In nice cases it is possible to construct a right localisation of $\C$ with respect to $K$. We state the general result \cite[Theorem 5.1.1]{Hir03} below.

\begin{theorem}[Hirschhorn]
Let $\C$ be a right proper cellular model category and $K$ a set of objects in $\C$. Then there exists a model structure $R_K\C$ on the underlying category $\C$ such that
\begin{itemize}
\item the weak equivalences in $R_K\C$ are the $K$-coequivalences
\item the fibrations in $R_K\C$ are the fibrations of $\C$.
\end{itemize}
\end{theorem}

One has to distinguish between $K$-cofibrant, cofibrant in $\C$ and $K$-colocal. Note that an object is $K$-cofibrant if and only if it is $K$-colocal and cofibrant in $\C$. The cofibrant replacement functor $Q_K$ of $R_K\C$ provides a colocalisation for an object $X$, that is, a $K$-coequivalence
\[
Q_K(X) \longrightarrow X
\]
with $Q_K(X)$ a $K$-colocal object of $\C$.

\begin{ex}\label{ex:cellular}
Let us again return to the example where $\C=\ch(R)$ and $A$ is a perfect $R$-module. In this special case, the cofibrant replacement $Q_A$ provides the \textbf{$A$-cellular approximation}
\[
\cell_A(M) \longrightarrow M.
\]
This means that $\cell_A(M)$ is ``built'' from $A$ using exact triangles and coproducts \cite[Section 4]{DwyGre02}. In this setting, cellular approximation satisfies
\[
\cell_A(M) \cong \cell_A(R) \otimes^L_{R} M,
\]
giving rise to the cofibrant replacement map
\[
\cell_A(R) \otimes^L_R M \stackrel{\sim}{\longrightarrow} M.
\]
Analogously to the definition of a smashing left localisation we can call this right localisation
\textbf{right smashing}: a right localisation of a monoidal model category $\C$ with unit $S$ is right smashing if
\[
Q_K S \otimes^L X \longrightarrow X
\]
is a $K$-cofibrant approximation for all $X$.
\end{ex}

\medskip
Dually to the local case we see that the class of $K$-coequivalences is closed under $\Omega$. Also, the class of $K$-colocal objects is closed under $\Sigma$.

\begin{definition}
Let $K$ be a class of objects in $\C$. We say that $K$ is \textbf{stable} if the class of $K$-colocal objects is also closed under $\Omega$.
\end{definition}

We also have the dual result to Lemma \ref{lem:stableconditions}:
if $\C$ is a stable model category, 
then a class of objects $K$ is stable if and only if
the collection of $K$-coequivalences objects is closed under $\Sigma$.
In particular if $K$ is closed under $\Omega$, then it is stable.

\begin{rmk}
As with remark \ref{rmk:usehomotopyclasses}, we see that if $K$ is a stable set of objects
then a map $f \co X \to Y$ is a $K$-coequivalence if and only if 
\[
[k,f]_*^{\C} \co [k,X]_*^{\C} \longrightarrow [k,Y]_*^{\C}
\] 
is an isomorphism of graded abelian groups for all 
$k \in K$. 
Similarly, $A$ is $K$-colocal if and only if for all $K$-coequivalences 
$f \co X \to Y$, the map  
\[
[A,f]_*^{\C} \co [A,X]_*^{\C} \longrightarrow [A,Y]_*^{\C}
\] 
is an isomorphism of graded abelian groups.
\end{rmk}

\begin{ex}
The case of $A$-torsion modules for a perfect $R$-module $A$ provides an example of a class of stable colocal objects.
\end{ex}

\begin{proposition}\label{prop:rightstable}
Let $\C$ be a stable model category and $K$ a stable class of objects. Assume that $R_K\C$ exists. Then $R_K\C$ is also stable. \qed
\end{proposition}

We omit the proof since it is very similar to the proof of Proposition \ref{prop:leftstable}.

\bigskip
We can always make a set of objects stable, but this usually changes the resulting 
model structure and homotopy category drastically. 
\begin{lemma}
Let $K$ be a class of cofibrant objects in a stable model category $\C$.
Define $\Omega^\infty K$ 
to be the collection of objects $Q{\Omega^n X}$ for $X \in K$
and $n \geqslant 0$. 
Then, provided it exists, 
$L_{\Omega^\infty K} \C$ is a stable model category.
Furthermore $K$ is stable if and only if 
$L_{\Omega^\infty K} \C$ is equal to $L_K \C$. \qed
\end{lemma}

We know that the right localisation of a right proper model category is again right proper 
by \cite[Theorem 5.1.5]{Hir03}. If $K$ is stable, then we also see that $R_K\C$ is left proper whenever $\C$ is. 

\begin{proposition}
Let $\C$ be a stable left proper model category. Let $K$ be a stable class of objects. If $R_K\C$ exists, then it is left proper.
\end{proposition}

\begin{proof}
Consider a pushout
\[
\xymatrix{ A \ar[d]_{f} \ar[r]^{p} & C \ar[d]^{g} \\
B \ar[r]_{q} & P
}
\]
where $p$ is a $K$-cofibration and $f$ is a $K$-coequivalence. 
We see immediately that $q$ is a $K$-cofibration. 
We would like to show that $g$ is a $K$-coequivalence. 

Since $\C$ is left proper, the cofibre of $p$ (the pushout of $p$ along 
$A \to \ast$) is also the homotopy cofibre $Cp$ of $p$.
Similarly, the cofibre of $q$ agrees with the homotopy cofibre $Cq$ of $q$. 
Since we have a pushout, the two cofibres are isomorphic, 
hence the map $c$ below is a weak equivalence in $R_K \C$. 
By Proposition \ref{prop:rightstable}, $R_K\C$ is stable, so the following is 
a morphism of exact triangles
\[
\xymatrix{ A \ar[r]\ar[d]^{f}_{\sim} & C\ar[r]\ar[d]^{g}  & Cp\ar[r]\ar[d]^{c}_{\sim} & \Sigma A \ar[d]^{\Sigma f}_{\sim}\\
B \ar[r] & P\ar[r] & Cq \ar[r]& \Sigma B.
}
\]
By the five-lemma for triangulated categories $g$ is a $K$-coequivalences.
\end{proof}

We know that $R_K\C$ has the same fibrations (and hence acyclic cofibrations) as $\C$ but fewer cofibrations. Generally, it is very hard to specify a set of generating cofibrations for $R_K\C$. However, if $\C$ and $K$ is stable, we are going to obtain a convenient description.

Following the previous section, a set of \textbf{horns} on $K$ is defined as
\[
\Lambda  K = \{ X \otimes \partial \Delta[n]_+ \longrightarrow 
X  \otimes \Delta[n]_+
\,\,|\,\, n \ge 0, X \in K\}.
\]
Remember that the operation $\otimes$ is defined via 
framings in $\C$ as in Section \ref{sec:modelcats}.
We have assumed that the set $K$ consists of cofibrant objects, 
so $\Lambda  K $ consists of cofibrations of $\C$. 

\begin{theorem}\label{thm:rightgencofibs}
Let $\C$ be a stable, right proper, cellular model category with a set of generating cofibrations $I$ and generating acyclic cofibrations $J$. 
Let $K$ be a stable set of  cofibrant objects.
Then $R_K\C$ is cellular with generating cofibrations $J \cup \Lambda K$ and acyclic cofibrations $J$.
\end{theorem}

\begin{proof}
We know that the model structure exists, is stable and is right proper.
We prove that $R_K \C$ is a cellular model category, 
via \cite[Theorem 12.1.9]{Hir03}. The various
smallness and compactness arguments follow from the 
corresponding statements for  $\C$ and the 
fact that $K$ consists of cofibrant objects. 

All that remains is to show that a map $f$ is a trivial $K$-fibration if and only if it has the right lifting property with respect to $J \cup \Lambda K$.
By \cite[Proposition 5.2.5]{Hir03} the maps of $J \cup \Lambda K$ are cofibrations
of $R_K \C$. Hence if $f$ is a fibration and $K$-coequivalence then $f$ has the right lifting property with respect to $J \cup \Lambda K$.
Now assume that $f$ has the right lifting property with respect to $J\cup\Lambda K$. Since $f$ has the right lifting property with respect to $J$, it is a fibration in $\C$ and hence it is a fibration in $R_K\C$. Now we want to show that $f$ is a $K$-coequivalence. 

By \cite[Proposition 5.2.4]{Hir03} a map $g: A \longrightarrow B$ 
with $B$ \emph{fibrant} has the right lifting property 
with respect to $J \cup \Lambda K$ if and only if 
$g$ is a fibration and a $K$-coequivalence. 
However, this is not true for general $B$ and we cannot simply assume $B$ to be fibrant.

However, we are working in a stable setting. Since $R_K\C$ is stable, $f$ being a $K$-coequivalence is equivalent to asking for its fibre (which in our setting is also its homotopy fibre) to be $K$-coacyclic. The fibre $F$ is the pullback of the diagram
\[
\ast \longrightarrow B \stackrel{f}{\longleftarrow} A.
\]
As $f$ has the right lifting property with respect to $J \cup\Lambda K$ and $F$ is a pullback, $F \longrightarrow \ast$ also has this right lifting property. The terminal object $\ast$ is fibrant, so by \cite[Proposition 5.2.4]{Hir03} $F$ is $K$-coacyclic, which is what we needed to prove. 
\end{proof}

\begin{rmk}\label{rmk:rightnoncellular}
Just as with Remark \ref{rmk:noncellular} we can replace the assumption that 
$\C$ is cellular with the assumption that $\C$ is cofibrantly generated
and the domains of $J \cup \Lambda K$ are small with respect
to the class of transfinite compositions of pushouts of $J \cup \Lambda K$. Thus, the theorem also 
provides a refinement of the general existence theorem of right localisations for the stable case.
\end{rmk}

The theorem is again an improvement on the general setting where $\C$ has not been assumed to be stable. 
Without stability, the results of \cite{Hir03} only prove the existence of some set
of generating cofibrations. Indeed, the set $J \cup \Lambda K$ is not always a generating
set of cofibrations for $L_S \C$, as shown by \cite[Example 5.2.7]{Hir03} which we will spell out now.

\begin{ex}
Consider the model category of pointed simplicial sets $\SSet_*$. 
Let $A$ be the quotient of $\Delta[1]$ obtained by identifying the 
the vertices of $\Delta[1]$. The geometric realisation of this simplicial
set is homeomorphic to the circle.  We consider the right localisation of $\SSet_*$ with respect to $K=\{A\}$, having one $0$-simplex and one $1$-simplex.

Let $Y$ be $\partial \Delta[2]$, whose geometric
realisation is also homeomorphic to the circle.  
Let $X$ be the simplicial set built from six $1$-simplices
with vertices identified so that the geometric realisation of $X$ is a circle. 
There is a fibration $p \co X \to Y$, whose geometric realisation
is the double covering of the circle. 

Now let $F(A,X)$ denote the simplicial set of maps from $A$ to $X$. We observe that $F(A,X)$ has only one simplex in each degree. The reason for this is the fact that
the only pointed map from $A$ to $X$ is the constant map to the basepoint. By induction, 
this also holds for maps from $A \smashprod \Delta[n]_+$ to $X$. The same is true for $F(A,Y)$, so 
\[
F(A,p): F(A,X) \longrightarrow F(A,Y)
\]
is an isomorphism.  

The map $p$ is a fibration, so it has the right lifting property with respect to $J$. 
The above argument shows that $p$ also has the right lifting property with respect to $\Lambda(A)$, hence it has the right lifting property with respect to $J \cup \Lambda(A)$. 

But $p$ is not a $K$-coequivalence as we shall show now. 
Consider the map below, which is induced by $p$
\[
\sing \hom(|A|,|p|): \sing \hom(|A|,|X|) \longrightarrow \sing \hom(|A|,|Y|)
\]
where $\hom(|A|,|X|)$ denotes the space of maps between the topological spaces $|A|$ and $|X|$ and $\sing$ the singular complex functor. However, since $|p|$ is a double cover of the circle, the map 
\[
\pi_0(\sing \hom(|A|,|p|)): \pi_0(\sing \hom(|A|,|X|)) \longrightarrow \pi_0 \sing (\hom(|A|,|Y|))
\]
is multiplication by $2$ on the integers. Thus, $\sing \hom(|A|,|p|)$ is not a weak equivalence.
Now we note that for any simplicial sets $P$ and $Q$, 
$\Map(P,Q)$ is naturally weakly equivalent to 
$\sing \hom(|P|,|Q|)$. Thus 
$\Map(A,p)$ is not a weak equivalence as claimed. 
\end{ex}

\section{Monoidal Left Localisations}\label{sec:monoidalleft}

Let $\C$ be a cellular and left proper model category and let
$S$ be a set of maps in $\C$.
Then we can ask the following: if $\C$ is monoidal, when is $L_S \C$ also
monoidal? When $\C$ is stable, we can use our preceding results to examine monoidality in a convenient way.

For this we need to know that $L_S \C$
satisfies the pushout product axiom. Recall that the \textbf{pushout-product} of two maps $f:A \rightarrow B$ and $g: C\rightarrow D$ is defined as
\[
f \square g: A \otimes D \coprod_{A \otimes C} B \otimes C \longrightarrow B \otimes D.
\]
A model category with monoidal product and unit $(\C, \otimes, \mathbb{S})$ is a \textbf{monoidal model category} if the pushout-product of two cofibrations is again a cofibration which is trivial if either $f$ or $g$ is. Further, then unit $\mathbb{S}$ of $\C$ has to satisfy a cofibrancy condition, see \cite[Definition 4.2.6]{Hov99}.

Thus, the usual method to examine monoidality of a model category is to examine its sets of
generating cofibrations and
generating acyclic cofibrations.
By Theorem \ref{thm:stableleftlocal} we know that if $\C$ is stable and 
proper and that $S$ is a stable set of
cofibrations between cofibrant objects, 
then the generating 
acyclic cofibrations of $L_S\C$ have the form $J \cup \Lambda S$.
Since $\C$ is assumed to be monoidal we know that $J \square I$ consists of weak equivalences in $\C$.
Thus if $\Lambda S \square I$ is a set 
of $S$-equivalences, then $L_S \C$ is monoidal.  
Conversely, if $L_S \C$ is monoidal, then 
$\Lambda S \square I$ consists of $S$-equivalences. 

Now we apply \cite[Theorem 5.6.5]{Hov99}, 
which essentially states that framings and monoidal products interact well,
to see that the image of the set 
$\Lambda S \square I$ in $\Ho(\C)$ is isomorphic to the image of 
the set $\Lambda (S \square I)$ in $\Ho(\C)$.
Thus $\Lambda S \square I$ consists of $S$-equivalences
if and only if 
$\Lambda (S \square I)$ consists of $S$-equivalences.
Furthermore $S \square I$ consists of $S$-equivalences if and only if 
$\Lambda (S \square I)$ consists of $S$-equivalences.
Hence we have the following result
and definition. 

\begin{lemma}\label{lem:square}
Let $\C$ be a proper, cellular and monoidal stable model category.
Let $S$ be a stable set of cofibrations between cofibrant objects. 
Then the set $S \square I$ is contained in the class of $S$-equivalences
if and only if $L_S \C$ is a monoidal model category. \qed
\end{lemma}

\begin{definition}
A stable set of cofibrations $S$ in a monoidal model category $\C$ is said to be 
\textbf{monoidal} if 
$S \square I$ is contained in the class of $S$-equivalences. 
\end{definition}

We can use this to restate a well-known fact.
\begin{ex}\label{ex:smodules}
The generating set $\mathcal{J}$ of $E_*$-equivalences in $\MSp$, 
the model category of EKMM $\mathbb{S}$-modules, is monoidal. 
This follows from the fact that if $f$ is an $E_*$-equivalence and $A$ is a cofibrant spectrum, 
then $f \otimes A$ is also an $E_*$-equivalence. 
Hence, by Lemma \ref{lem:square}, $L_E (\MSp)$ is a monoidal model category.
\end{ex}

\begin{lemma}\label{lem:squaremonoidal}
Let $\C$ be a proper, stable, cellular, monoidal model category. 
Assume that $S$ is a stable set of cofibrations between cofibrant objects. 
Then $S \square I$ is a monoidal stable set of maps. 
Hence $L_{S \square I} \C$ is a stable monoidal model category in which 
the maps $S$ are weak equivalences. 
\end{lemma}
\begin{proof}

Take any $s \in S$ and any cofibration $a$.  
Then the map
$a$ is a retract of pushouts of transfinite compositions
of maps in $I$. Hence $s \square a$ 
a retract of pushouts of transfinite compositions
of maps in the set $S \square I$. 
Thus $s \square a$ is an $S \square I$-equivalence. 

We need to check that $S \square I$ is still stable, 
so consider some $s \square i$. 
Let $\sphspec^{-1}$ be some cofibrant desuspension of the unit
$\sphspec$ of $\C$. 
We know that $(s \square i) \otimes \sphspec^{-1}$ is
isomorphic to $s \square (i \otimes \sphspec^{-1})$, 
which, by the above, is an $S \square I$-equivalence.
It follows immediately that the $S \square I$-equivalences
are closed under desuspension, so our set is stable. 

Now we must check that $(\Lambda(S \square I)) \square I$
consists of $S \square I$-equivalences. But every element in $\Lambda( S \square (I \square I ))$ is weakly equivalent to an element in $(\Lambda(S \square I) \square I)$. We know that any map in
$S \square (I \square I)$ is an $S \square I$-equivalence and a cofibration. 
Furthermore a horn on such a map is still an $S \square I$-equivalence.

Finally, to see that $S$ consists of $S \square I$-equivalences, 
consider the cofibration $\eta \co \ast \to Q \sphspec$. 
For any $s \in S$, $s \square \eta$ is isomorphic to 
$s \otimes Q \sphspec$, which is weakly equivalent to $s$ 
since the domains and codomains of $S$ are cofibrant. 
\end{proof}

We may also conclude that if $S$ is monoidal, then 
$L_{S \square I} \C$ is equal to $L_S \C$. 
Usually, however, localising at 
${S \square I}$ and $S$ give different model categories. 
While the above result makes more maps into weak equivalences
than we might want, it actually does so in quite a minimal way, as the result
below shows. We can think of this as saying that 
$L_{S \square I} \C$ is the \emph{monoidal} left Bousfield localisation of 
$\C$ at the stable set $S$.

\begin{theorem}\label{thm:monleftlocal}
Let $$F: \C \lradjunction \D: G$$ be a lax monoidal Quillen pair between monoidal model categories $\C$ and $\D$.
Assume that $\C$ is proper, stable and cellular. Let $S$ be a stable set of
cofibrations between cofibrant objects in $\C$. If $F(s)$ is a weak
equivalence in $\D$ for all $s \in S$, then this adjoint pair factors
uniquely over the change of model structures adjunction between $\C$ and
$L_{S \square I} \C$.
That is, we have a commutative diagram of left adjoints of weak monoidal 
Quillen pairs 
\[
\xymatrix{
\C \ar[rr]^{F} \ar[dr]_{\id} && \D \\
& L_{S \square I} \C \ar[ur]_{\bar{F}}
}
\]
\end{theorem}
\begin{proof}
We must show that the image under $F$ of every element $S \square I$ is an isomorphism in 
$\Ho(\D)$. Take some $i \in I$ and replace it by a weakly equivalent map $j$
that is a cofibration between cofibrant objects. 
We must show that for any $s \in S$,
$F(s \square j)$ is a weak equivalence in $\D$. 
We have a weak monoidal Quillen pair and the domain and 
codomain of $s \square j$ are cofibrant. Thus we see that 
$F(s \square j)$ is weakly equivalent to 
$Fs \square Fj$. Since $\D$ is monoidal, 
this is an acyclic cofibration of $\D$. 

Hence we have the desired factorisation of Quillen functors via the universal 
property of left Bousfield localisations \cite[Definition 3.1.1]{Hir03}. 
Furthermore, the thus obtained $\bar{F}$ and its right adjoint $\bar{G}$ form a lax monoidal 
Quillen pair between the monoidal model categories
$L_{S \square I} \C$ and $\D$. 
\end{proof}

If we restrict ourselves to spectra, then we can use the above to obtain a 
very concise description of the generating sets of a monoidal
stable localisation. For the result below we can use EKMM $\mathbb{S}$-modules, 
symmetric spectra, orthogonal spectra or their equivariant versions for a compact Lie group.

\begin{proposition}
Let $\Sp$ be a monoidal model category of spectra from the list above.
Let $S$ be some set of cofibrations between cofibrant objects in $\C$. Then $L_{S \square I} \Sp$ exists, 
is cellular, proper, stable and monoidal. It has generating sets given by $I$ and 
$J \cup (S \square I)$. 
\end{proposition}

\begin{proof}
The model category of spectra $\Sp$ comes equipped with a collection of 
evaluation functors 
\[
U_V: \C \longrightarrow \SSet_*
\]
for each $V$ of the indexing category (such as the non-negative integers
or finite dimensional real inner product spaces). Let $F_V$ be the left adjoint to
$U_V$. 

We see that the set of generating cofibrations $I$ of $\Sp$ can be chosen to consist of maps of the form 
$F_V l$, where $l$ is some generating cofibration for 
simplicial sets. It follows immediately that 
$S \square I$ is stable. Hence, by Lemma \ref{lem:squaremonoidal} we know that $L_{S \square I}\Sp$ is monoidal. By the results of Section \ref{sec:stableleftloc} we know that it is also stable, proper and cellular and cofibrantly generated by
 $I$ and $J \cup \Lambda (S \square I)$.
 
We now need to show that a map $f$ has the right lifting property with respect to 
$\Lambda (S \square I)$ if and only if it has the right lifting property with respect to 
$S \square I$. 

Because $\Sp$ is a simplicial model category, we can assume that an element in $\Lambda (S \square I)$ 
is of the form form $(s \square F_V l) \square k$, where $l$ and $k$ are generating cofibrations for 
simplicial sets and $s \in S$.
But this is isomorphic to $S \square F_V(l \square k)$. 
It follows that the sets $\Lambda(S\square I)$ and $S \square (I \square I)$ agree. 

We have $I \subseteq I \square I$ because $\iota \square F_V l =  F_V l$ where 
\[
\iota = (F_0 \partial \Delta[0]_+ \longrightarrow F_0 \Delta[0]_+).
\]
Thus,
\[
S \square I \subseteq S \square (I \square I) = \Lambda( S \square I)
\]
and hence if we define $A\mbox{--cof}$ to be the class of maps with the left lifting property
with respect to all maps with the right lifting property with respect to $A$, 
then we see that 
\[
(S\square I)\mbox{--cof} \subseteq \Lambda (S \square I)\mbox{--cof}.
\]

For the other inclusion, we know that model category $\Sp$ is monoidal, so
$I \square I \subseteq I\mbox{--cof}$. Thus
\[
\Lambda( S \square I) = S \square (I \square I) 
\subseteq S \square(I\mbox{--cof}) \subseteq (S \square I)\mbox{--cof}.
\]
\end{proof}

Recall from Section \ref{sec:monoidalleft} that in the case of a smashing localisation we have $L_E \Sp = L_\Gamma \Sp$ for
\[
\Gamma = \{ \Sigma^n \lambda: \Sigma^n \sphspec \longrightarrow \Sigma^n L_E \sphspec \,\, |\,\,n \in \mathbb{Z}\}.
\]
Together with Corollary \ref{cor:smashingcofibs} we achieve the following.

\begin{corollary}\label{cor:spectragenerating}
Let $\Sp$ be a monoidal model category of spectra with generating cofibrations $I$ and acyclic cofibrations $J$. Let $L_E$ be a smashing Bousfield localisation. Then $L_E\Sp$ is proper, cellular, stable and monoidal with generating cofibrations $I$ and generating acyclic cofibrations $J \cup (\Gamma \square I)$. \qed
\end{corollary}

\section{Monoidal Right Localisations}\label{sec:monoidalright}

Let $\C$ be a cellular and right proper model category and let
$K$ be a set of objects in $\C$.
Then we can ask the following: if $\C$ is monoidal, when is $R_K \C$ also
monoidal? We can use our preceding work on stability and generating cofibrations to give a compact and useful answer. We will then apply this to some examples.

We start with an observation. Recall that an object in $\C$ is $K$-cofibrant if and only if it is $K$-colocal and cofibrant in $\C$.
The elements of $K$ are $K$-cofibrant. Thus, if $R_K \C$ is monoidal, then any element of form 
$k \otimes k'$ for $k, k' \in K$ will also be $K$-cofibrant. We show that this
necessary condition is almost sufficient for monoidality of $R_K\C$.

\begin{definition}\label{def:monoidalset}
Let $K$ be a set of cofibrant objects in a right proper, cellular, monoidal model category $\C$. 
We say that $K$ is \textbf{monoidal} if the following two conditions hold.
\begin{itemize}
\item Any object of the form $k \otimes k'$,
for $k, k' \in K$, is $K$-colocal. 
\item For $Q_K \mathbb{S}$ a $K$-cofibrant replacement of the
unit $\mathbb{S}$ of $\C$ and any $k \in K$, the map 
$Q_K \mathbb{S} \otimes k \to k$ is a $K$-coequivalence. 
\end{itemize}
\end{definition}

Note that if the first condition holds, then the
domain and codomain of $Q_K \mathbb{S} \otimes k \to k$ 
are both $K$-cofibrant. Hence, this map is a $K$-coequivalence if and only
if it is a weak equivalence of $\C$. 
Obviously, if the monoidal unit is an element of $K$, then the second condition
holds automatically. 

Recall that a model category satisfies the \textbf{monoid axiom} if all 
transfinite compositions of pushouts of maps of the from 
$j \otimes Z$, for $j$ an acyclic cofibration and 
$Z$ any object of $\C$, are weak equivalences. 
This is a very useful tool for considering the category of modules over a monoid $R$ in $\C$: if $\C$ is cofibrantly generated, monoidal and satisfies the monoid axiom  (and some smallness assumptions hold), then the category of $R$-modules in $\C$ is also a cofibrantly generated model category by \cite[Theorem 4.1]{SchShi00}.

\begin{theorem}
Let $\C$ be a stable, proper, 
cellular and monoidal model category. Let $K$ 
be a stable collection of cofibrant objects. 
Then $R_K \C$ is monoidal if and only if $K$ is monoidal.

Further, if $K$ is monoidal and $\C$ also satisfies the monoid axiom, 
then so does $R_K \C$.
\end{theorem}

\begin{proof}
If $R_K \C$ is monoidal, then the pushout product axiom implies that $K$ is monoidal. 
For the converse, assume that $K$ is monoidal.  To show that $R_K\C$ is monoidal, we must verify the two conditions of \cite[Definition 4.2.6]{Hov99}. The second of these, namely that 
$$Q_K \mathbb{S} \otimes k \to k$$ 
is a $K$-coequivalence, holds by assumption.

\medskip

Remember from Theorem \ref{thm:rightgencofibs} that $R_K\C$ has generating cofibrations $\Lambda K \cup J$ and acyclic cofibrations $J$. 
Hence, we must check that $(\Lambda K \cup J) \square (\Lambda K \cup J)$
consists of $K$-cofibrations. 
This amounts to proving that the following three collections $\Lambda K \square \Lambda K$, 
$\Lambda K \square J$ and $J \square J$ consist of $K$-cofibrations.
For the first, consider \[
i = (\partial \Delta[n]_+ \otimes k 
\to \Delta[n]_+ \otimes k) \square 
(\partial \Delta[m]_+ \otimes k' 
\to \Delta[m]_+ \otimes k') \in \Lambda K \square \Lambda K
\]
which is a cofibration in $\C$ since $\C$ was assumed to be monoidal.
We can rewrite $i$, up to weak equivalence, as 
the following map which is a cofibration of $\C$
between $K$-colocal objects.
\[
\left( (\partial \Delta[n]_+ \to \Delta[n]_+) \square
(\partial \Delta[m]_+ \to \Delta[m]_+) \right) \otimes 
(k \otimes k')
\]
Thus the domain and codomain
of $i$ are $K$-colocal, so by 
\cite[Proposition 3.3.16]{Hir03}
$i$ is also $K$-cofibration. 

Let us now look at the second collection, $\Lambda K \square J$.
A map in this set is contained in the class of maps
$I \square J$--cof, which consists of acyclic cofibrations of $\C$.
Any such map is a $K$-cofibration.
The same argument holds for the third collection, $J \square J$. Thus, the pushout-product of two $K$-cofibrations is again a $K$-cofibration which is acyclic if either of the two maps is. 

The monoid axiom holds in $R_K \C$ if it holds in $\C$, 
since the set of generating
acyclic cofibrations has not changed. 
\end{proof}

We can apply this to Dwyer-Greenlees' example of right Bousfield localisation, where $\C=\ch(R)$ and $K=\{A\}$ a perfect $R$-module, see Section \ref{sec:examples}.

\begin{corollary} The model category $R_{\{A\}}(\ch(R))$ of $A$-torsion $R$-modules is a monoidal model category.
\end{corollary}

\begin{proof} We consider $\ch(R)$ with the projective model structure. Since $A$ is a perfect chain complex of $R$-modules, it is of finite length and is degreewise projective. Hence $A$ is cofibrant in $\ch(R)$. We are now going to check that $K=\{A\}$ satisfies the two conditions of Definition \ref{def:monoidalset}. 

We remember from Example \ref{ex:cellular} that in this case the cofibrant replacement is the same as cellular approximation and that cellular approximation is given by the weak equivalence
\[
\cell_A(R) \otimes^L_{R} M \longrightarrow \cell_A(M).
\]
For the unit condition we must prove that 
\[
\cell_A(R) \otimes^L_R \cell_A(M) \longrightarrow \cell_A(M)
\] 
is an $\{A\}$-coequivalence for any $M$. 
But this map is simply cellular approximation of a 
cellular object, hence it is a weak equivalence. 

We now have to check that $A \otimes A$ is $\{A\}$-colocal. For this we have to show that
\[
\Map_{\ch(R)}(A \otimes A,N) \simeq * \,\,\,\,\mbox{for any $N$ with}\,\,\,\,\Map_{\ch(R)}(A,N) \simeq *.
\]
But in this case, $\Map_{\ch(R)}(X,Y) \simeq *$ is equivalent to $R\Hom_R(X,Y) = 0$ as
\[
\pi_k(\Map_{\ch(R)}(X,Y)) \cong [S^0, \Map_{\ch(R)}(\Sigma^{-k}X,Y)] \cong R\Hom_R^{-k}(X,Y).
\]
We also have by adjunction
\[
R\Hom_R(A \otimes A, N) \cong R\Hom_R(A, R\Hom_R(A,N)),
\]
so our claim follows.
\end{proof}

\bigskip
Just as we may make any set of objects $K$ stable, we may also make any stable set
into a monoidal stable set. Let $\bar{K}$ denote the collection of objects
$k_1 \otimes k_2 \dots \otimes k_n$ for all $n \geqslant 0$, 
with the zero-fold product being the cofibrant replacement of the unit. This set is clearly monoidal so 
$R_{\bar{K}} \C$ is a monoidal model category. 
However, $R_{\bar{K}} \C$ has \emph{fewer}
weak equivalences, so in general a $K$-coequivalence
is not a $\bar{K}$-coequivalence. So this notion of replacing $K$ by 
$\bar{K}$ is perhaps less useful than 
the version for left localisations. 

Dually to Theorem \ref{thm:monleftlocal} we can show that $R_{\bar{K}}\C$ is the best we can achieve. The following result essentially says that $R_{\bar{K}}$ is the ``closest'' right localisation to $R_K\C$ for arbitrary stable $K$ that is also monoidal.

\begin{proposition}
Let $\C$ be a right proper, stable, cellular 
monoidal model category. Then the identity adjunction gives Quillen pairs as below
where the right hand adjunction is a monoidal Quillen pair.
\[
R_K \C \lradjunction 
R_{\bar{K}} \C \lradjunction
\C
\]
\end{proposition}

\begin{proof}
Every object of $K$ is cofibrant in $\C$. Since $\C$ is monoidal, every object of the form $k_1 \otimes k_2 \otimes \cdots \otimes k_n$ for $k_i \in K$ and $n \ge 0$ is also cofibrant in $\C$. It follows that $R_K \C \lradjunction \C$ factors over $R_{\bar{K}}\C$ as required, giving a monoidal Quillen pair $R_{\bar{K}}\C \lradjunction \C$.
\end{proof}

\section{Replacing stable model categories by spectral ones}\label{sec:spectral}

Model categories are fundamentally linked to simplicial sets
via framings. 
But framings are only well behaved on the homotopy category.
For many tasks it is preferable to have a simplicial model category. 
Hence the question of when is a model category Quillen equivalent to a simplicial one?
The paper \cite{Dug01} provides an answer to this question. 

Stable model categories are fundamentally linked to 
spectra via stable framings, see \cite{Len12}. 
Stable framings are even more poorly behaved on the model category level
than framings. Hence we would like an answer to the question:  
when is a model category Quillen equivalent to a spectral one?

\begin{definition}
A \textbf{spectral model category} is a model category that is enriched, tensored and cotensored over symmetric spectra. Further, it
 satisfies the analogue of 
Quillen's SM7 with simplicial sets replaced by symmetric spectra $\Sigma \Sp$. 
In the language of \cite[Definition 4.2.18]{Hov99} it is a $\Sigma \Sp$-model category.
\end{definition}

We can now use our work on left localisations to weaken the known assumptions that a model category has to satisfy in order to be Quillen equivalent to a spectral one. Because of Proposition \ref{prop:rightproper} we can now combine results from Dugger and Schwede-Shipley to acquire the following result.

\begin{theorem}\label{thm:specreplace}
If $\C$ is a model category that is stable, proper and cellular,
then it is Quillen equivalent to a 
spectral model category that is also
stable, proper and cellular.
\end{theorem}
\begin{proof}
Because $\C$ is cellular and left proper, 
\cite[Theorem 1.2]{Dug01} states that $\C$ is 
Quillen equivalent to a simplicial model category. 
Specifically, $\C$ is Quillen equivalent to a
non-standard model structure on the category of simplicial objects
in $\C$, which we write as $s \C_{hc}$. 

In more detail, one starts by equipping the category of simplicial objects
in $\C$ with the Reedy model structure. 
A Reedy weak equivalence is a map of simplicial objects $f \co A \to B$ 
such that on each level $f_n$ is a weak equivalence of $\C$. 
Every Reedy cofibration is a levelwise cofibration and 
every Reedy fibration is a levelwise fibration, see  \cite[15.3.11]{Hir03}.
It follows immediately that $s \C$ is still stable. 
Since $\C$ is cellular and proper, so is $s \C$
by \cite[Theorems 15.7.6 and 15.3.4]{Hir03}.

The model category $s \C_{hc}$ is defined as a left Bousfield
localisation of $s \C$ at a set $S$ of maps defined 
just above Theorem 5.2 in \cite{Dug01}. Since 
$s \C_{hc}$ is Quillen equivalent to $\C$, it must also be stable. 
Hence by Proposition \ref{prop:rightproper}, $s \C_{hc}$ is right proper. 
Thus we now know that $s \C_{hc}$ is a proper, cellular, stable model category. 

We now use the results of \cite{SchShi03} to replace this by a Quillen
equivalent spectral model category. 
We rename $s \C_{hc}$ as $\D$ and denote 
the category of symmetric spectra in $\D$, 
by $\Sigma \Sp (\D, S^1)$.
We can equip this category with the levelwise (or projective) model structure,  
where fibrations
and weak equivalences are defined levelwise. 
This model structure is cellular, proper and stable. 

We then left localise the model structure at a set of cofibrations
to obtain the `stable' model structure on $\Sigma \Sp (\D, S^1)$. 
By \cite[Theorem 3.8.2]{SchShi03} this model structure is spectral and 
there is a Quillen equivalence between $\D$ and $\Sigma \Sp(\D, S^1)$
equipped with the stable model structure.
Our previous results also show that this stable model structure
on $\Sigma \Sp (\D, S^1)$ is proper.
\end{proof}

Results along this line have been proven in 
\cite{Dug06}. In that paper it is shown that a 
stable, presentable model category is Quillen equivalent
to a spectral model category. 
We replace the notion of presentable (which essentially means 
Quillen equivalent to a combinatorial model category)
with the more familiar notion of cellular. While we have to add
proper to our list of assumptions, our method of replacing a model 
category by a spectral one involves no choices and requires much less
technical work to understand the resulting category and model structure.

\section{Right localisation and Morita theory}\label{sec:morita}

In \cite[Theorem 2.1]{DwyGre02}, Dwyer and Greenlees show that the category of $A$-torsion $R$-modules (with $A$ a perfect $R$-module) is equivalent to the derived category of the ring $\End_R(A)$. In this section we are going to prove a more general version of this, namely that for a set of well-behaved objects $K$, the model category $R_K\C$ is Quillen equivalent to the category of modules over the endomorphism ring spectrum with several objects $\rightmod \End(K)$.

We say that an object $X$ in a stable model category $\C$
is \textbf{homotopically compact} if for any family of object $\{ Y_a \}_{a \in A}$
the canonical map below is an isomorphism. 
\[
\bigoplus_{a \in A} [X, Y_a]^\C \to [X, \coprod_{a \in A} Y_a]^\C
\] 
Homotopically compact objects have obvious technical advantages over general ones,
so it is natural to ask what happens if one right localises at a set of 
homotopically compact objects. 
We show that, with some minor assumptions, such right localisation are well understood, 
and we identify their homotopy categories.  

Let $\C$ a stable, cellular right proper spectral model category and 
let $K$ be a stable set of homotopically compact cofibrant-fibrant objects
of $\C$. The assumption that 
$\C$ be spectral is less demanding than it appears, by Theorem \ref{thm:specreplace}. 

Define $\End(K)$ to be the category enriched over symmetric spectra
with objects set given by $K$ and morphism spectra
given by $\hom(k,k')$ defined using the enrichment of $\C$ in 
symmetric spectra. Consider the category of 
contravariant enriched functors from $\End(K)$ to symmetric spectra, 
with morphisms the enriched natural transformations. 
We call this category $\rightmod \End(K)$.
It has a model structure with weak equivalences and 
fibrations defined termwise, see
\cite[Theorem A.1.1]{SchShi03}.

There is a Quillen pair 
\[
\rightmod\End(K)  \lradjunction \C \]
whose right adjoint takes
$X \in \C$ to $\hom(-,X)$ in $\mbox{mod-}\End(K)$. 
We call this right adjoint $\hom(K, -)$ and 
we write $- \smashprod_{\End(K)} K$  for its left adjoint. 

We are almost ready to start relating $\rightmod \End(K)$ and $R_K \C$, 
but we first need a technical result.

\begin{lemma}
Let $\C$ a stable, cellular right proper spectral model category 
and let $K$ be a stable set of cofibrant objects in $\C$. 
Then $R_K \C$ is a spectral model category. 
\end{lemma}

\begin{proof}
Since $\C$ is spectral, all we must show is the spectral analogue of (SM7), namely
that if $a$ is a cofibration of $R_K \C$
and $i$ is a cofibration of $\Sigma \Sp$, then 
$a \square i$ is a cofibration of $R_K \C$. 
It suffices to prove this for $a \in \Lambda K$ and $i$ a generating
cofibration of $\Sigma \Sp$. We know that $a \square i$
is a cofibration of $\C$. We must show that it is in fact a $K$-cofibration. 

Consider a generating cofibration $i$. It is of the form $F_n A \to F_n B$ for $A$ and $B$ simplicial sets
and $F_n$ the left adjoint to evaluation at level $n$. 
If $X \in \C$ is $K$-colocal then $X \otimes F_n A$
is weakly equivalent to $(\Sigma^{-n} X) \otimes A$. 
Since $K$ is stable, $\Sigma^{-n} X$ is $K$-colocal and
hence so is $(\Sigma^{-n} X) \otimes A$. 
It follows that the domain and codomain of $a \square i$
are both $K$-colocal. 
By \cite[Proposition 3.3.16]{Hir03}
a cofibration between $K$-colocal
objects is a $K$-cofibration. 
Hence $a \square i$ is a $K$-cofibration, which is what we wanted to prove.
\end{proof}

We need some new terms in order to state the main result of this section.
\begin{definition}
Let $\C$ be a stable model category. A full triangulated subcategory
of $\Ho(\C)$ with shift and triangles induced from $\Ho(\C)$ is called 
\textbf{localising} if it is closed under
coproducts in $\Ho(\C)$. A set $P$ of objects of $\Ho(\C)$ 
is called a set of \textbf{generators} if the only localising
subcategory which contains the objects of $P$ is $\Ho(\C)$ itself.
\end{definition}

\begin{theorem}\label{thm:morita}
Let $\C$ a stable, cellular right proper spectral model category and 
let $K$ be a stable set of cofibrant-fibrant objects
of $\C$. Then the Quillen pair 
\[
- \smashprod_{\End(K)} K : \rightmod \End{(K)} 
\lradjunction 
\C : \hom(K, -) 
\]
factors over 
$R_K \C$. 
Hence one has a diagram of Quillen pairs as below.
\[
\xymatrix@C+1.5cm@R+0.5cm{
\rightmod \End{(K)} 
\ar@<+1ex>[r]^(0.6){- \smashprod_{\End(K)} K} 
& R_K \C 
\ar@<+0ex>[l]^(0.4){\hom(K,-)} 
\ar@<+0.5ex>[r]^{\id} 
& \C
\ar@<+0.5ex>[l]^{\id} 
}
\]

If the set $K$ consists of homotopically compact 
objects, then the left hand Quillen pair in this diagram is a Quillen equivalence. 
Furthermore, the homotopy category of $R_K \C$ is 
triangulated equivalent to the 
localising subcategory of $\Ho(\C)$ generated by $K$. 

\end{theorem}

\begin{proof}
A generating cofibration of $\mbox{mod-} \End{K}$ 
takes form $\hom(-,k) \smashprod i$ where $i$ is a generating cofibration in symmetric spectra, $\smashprod$ is the smash product in symmetric spectra and $\hom(-,k) \in \mbox{mod-} \End{K}$. 
The functor ${- \smashprod_{\End{K}} K} $ sends this to
$k \smashprod i$, which is a cofibration of the spectral model category $R_K \C$.
Hence we have a factorisation of the Quillen functors as above. 

It is easy to check that if $k$ is compact in $\C$, 
then it is also compact in $R_K \C$. 
The set of cofibres of $\Lambda K \cup J$  (the generating cofibrations for 
$R_K \C$) is a generating set for the homotopy category of $R_K \C$. 
Since the cofibres of $J$ are contractible, we may ignore these. 
The cofibres of the sets $\Lambda K $ are simply 
suspensions of $K$ up to weak equivalence, hence it follows that 
$K$ is a generating set for the homotopy category of $R_K \C$.
We now apply 
\cite[Theorem 3.9.3]{SchShi03} to see that 
we have a Quillen equivalence and that 
the statement on homotopy categories holds. 
\end{proof}

Thus we have shown that in good circumstances a right localisation
is Quillen equivalent to the simpler notion of modules over an
endomorphism ringoid. In this setting we can identify 
$\Ho R_K \C$ as the smallest localising 
subcategory of $\C$ containing $K$.
Hence it is perfectly correct to think of $R_K \C$ as modelling the homotopy
theory of objects of $\C$ built from $K$ via coproducts, shifts and triangles. 
Thus right localisation in these circumstances simply alters which objects we think 
of as generators for the homotopy category. 
We also obtain an explicit description of $K$-colocalisation. If $X$ is fibrant in $\C$, 
then $K$-colocalisation is given by
\[
\hom(K,X) \smashprod_{\End{K}}^L K \longrightarrow X.
\]

This leads to questions for future research: if the set $K$ is not homotopically compact, 
how well does $R_K \C$ model $\rightmod \End(K)$? Similarly, 
if $\C$ is spectral but not cellular or right proper, 
and $K$ is a stable set of homotopically compact objects, 
how well does $\rightmod \End(K)$ model $R_K \C$, which may not exist?

\begin{ex}\label{ex:dgendomorphisms}
One half of \cite[Theorem 2.1]{DwyGre02} 
is the statement that the category of $A$-torsion $R$-modules
is equivalent to the derived category of modules over 
$\End_R(A)$, for $A$ a perfect complex. 
We are now able to give a model category level version of that result:
the right localisation of $\ch(R)$ at the perfect complex $A$ is 
Quillen equivalent to $\rightmod \End_R(A)$. 
\end{ex}

\bigskip
We are now going to use a duality argument to show that in some special cases, $R_K\C$ is Quillen equivalent to a left localisation of $\C$ at a set of maps $S$. In particular this applies to the case of $A$-torsion $R$-modules. For the rest of this section assume that $\C$ is a stable model category whose homotopy category $\Ho(\C)$ is monoidal with product $\wedge$ and unit $\mathbb{S}$.
Further, we ask for $\mathbb{S}$ being a homotopically compact generator. We also assume that $\Ho(\C)$ is closed in the sense that it possesses function objects $F(-,-)$. 
For example, $\C$ can be any smashing localisation of EKMM $\mathbb{S}$-modules. 

Remember that $X \in \Ho(\C)$ is said to be \textbf{strongly dualisable} if the natural map
\[
F(X,\mathbb{S}) \wedge Y \longrightarrow F(X,Y)
\]
is an isomorphism for all $Y$, see \cite[Definition 1.1.2]{HovPalStr97}. 
In our setting the class of homotopically compact objects is equal to 
the class of strongly dualisable objects by \cite[Theorem 2.1.3]{HovPalStr97}. 

Let $K$ be a set of objects in $\C$. By $DX := F(X,\mathbb{S})$ we denote the dual of an object $X$. Further, we define
\[
DK := \coprod_{k \in K} Dk
\]
\begin{definition}
We say that a morphism $f: X \longrightarrow Y$ in $\C$ is a $DK_*$-equivalence if 
\[
DK \wedge f: DK \wedge X \longrightarrow DK \wedge Y
\]
is an isomorphism in $\Ho(\C)$. 
\end{definition}

It is now easy to prove the proposition below, which we 
combine with Theorem \ref{thm:morita} to obtain the subsequent corollary.

\begin{proposition}
Let $K$ be a set of homotopically compact cofibrant objects in $\C$. Then the class of $K$-coequivalences is precisely the class of $DK_*$-equivalences. Furthermore, if $L_{DK_*}\C$ exists, then the identity functors provide a Quillen equivalence
\[
R_K\C \lradjunction L_{DK_*}\C. 
\] {\qed}
\end{proposition}

\begin{corollary}
Let $\C$ be a monoidal, stable, cellular, proper, spectral model category
with unit $\sphspec$ a homotopically compact generator. 
If $K$ is a set of homotopically compact cofibrant-fibrant objects in $\C$ such that $L_{DK_*}\C$ exists. Then the model categories $R_K\C$, $L_{DK_*}\C$ and $\rightmod \End(K)$ are Quillen equivalent.
\end{corollary}

This can be applied to the special case of $A$-torsion and $A$-complete $R$-modules for a perfect $R$-module $A$, obtaining Theorem 2.1 of \cite{DwyGre02}. In this case, we consider $A$-torsion modules $R_A\ch(R)$ and $A$-complete $R$-modules $L_{DA_*}\ch(R)$. Hence we recover Dwyer and Greenlees' result that $A$-torsion and $A$-complete $R$-modules are Quillen equivalent.

We can further specify to the case of $R=\mathbb{Z}$ and $A=(\mathbb{Z}\xrightarrow{\cdot p} \mathbb{Z}) \cong \mathbb{Z}/p$. In this case we obtain that $DA \cong A[1]$. Since $DA_*$-equivalences form a stable set, we recover Dwyer and Greenlees' ``paradoxical'' result that left and right localisation at $\mathbb{Z}/p$ agree.

\section{A correspondence between left and right localisations}\label{sec:correspondence}

We now turn to comparing left and right localisations. We show how given any left localisation, there is a 
corresponding right localisation governing the information of this left localisation and vice versa. 

\begin{lemma}
Let $\C$ be a cellular, proper, stable model category and $S$ be a stable set of maps in $\C$. Now let $T$ be the set of maps
$\ast \to Cs$, where $s \in S$ and $Cs$ is the cofibre of $s$. 
Then $T$ is a stable set of maps and $L_S \C = L_T \C$. 
\end{lemma}
\begin{proof}
Consider the exact triangle in $\Ho(\C)$
\[
X \stackrel{s}{\longrightarrow} Y \longrightarrow \Sigma Cs \longrightarrow \Sigma X
\]
for $s \in S$. Applying the graded homotopy classes of maps functor 
$[-,Z]^{\C}_*$ gives a long exact sequence. 
Remark \ref{rmk:usehomotopyclasses} now proves the claim.
\end{proof}

One advantage of replacing $S$ by the set $T$ is that we can 
see that the generating cofibrations for $L_S \C$ can be taken to be the set 
$\Lambda T \cup J$ where 
\[
\Lambda T = 
\{ Cs \otimes \partial \Delta[n]_+ \longrightarrow 
Cs  \otimes \Delta[n]_+
\,\,|\,\, n \ge 0, s \in S \}.
\]
We also see that $S$ is monoidal if and only if $T$ is monoidal, 
which might be easier to check in practice. Thus localising at 
$S$ is the same as making the set of objects $Cs$ acyclic. 
This is why left localisations are sometimes known as 
{acyclicisations}. 

Another advantage is that this description of left localisation
illuminates the relation between left and right localisations. 
Let $\C$ be a cellular, proper, stable model category with generating sets $I$ and $J$ 
and let $K$ be a stable set of cofibrant objects of $\C$. Then we can see that the difference
between left and right localising is whether to take $\Lambda K \cup J$ as the 
set of generating acyclic cofibrations or the set of generating cofibrations. 
This is the model category version of choosing to declare a set of objects to be trivial,
or declaring a set of objects to be generators.

\begin{definition}
For a set of maps $S$, define a set of objects 
$K_S = \{Cs \ | \ s \in S \}$. Conversely, given a set of objects $K$
define a set of maps $S_K : = \{ \ast \to k \  | \ k \in K  \}$. 
\end{definition}

Clearly, if $S$ is stable, then so is $K_S$. Similarly, 
if $K$ is stable, so is $S_K$. 
We immediately see that right localising at the set $K_{S_K}$ is the same as 
right localising at the set $K$. Similarly, left localising at $S_{K_S}$
gives the same model category as left localising at $S$.

\begin{proposition}\label{prop:orthogonal}
Choose some stable set of cofibrations $S$ and let $K=K_S$ 
or choose a set of cofibrant objects $K$ and let $S=S_K$.
Assume that $\C$ is stable, proper and cellular. Then there is a 
diagram of Quillen pairs
\[
R_K \C
\lradjunction
\C
\lradjunction
L_S \C
\]
such that the composite adjunction 
$\Ho(R_K \C) \lradjunction \Ho(L_S \C)$
is trivial in the sense that both functors send every object
to $\ast$. 
\end{proposition}

\begin{proof}
Every object in $\Ho(R_K\C)$ is isomorphic to a $K$-colocal object while every object in $\Ho(L_S\C)$ is isomorphic to an $S$-local one. By construction, being $K_S$-colocal is equivalent to being $S$-acyclic and being $K$-colocal is equivalent to being $S_K$-acyclic.
\end{proof}

The above adjunctions give a decomposition of the homotopy category of $\C$ into
two pieces which are \textbf{orthogonal}, in the sense that 
if $A$ is $K$-colocal and $Z$ is $S$-local, 
then $[A,Z]^\C =0$.
More clearly, the $K$-colocal objects
are precisely the $S$-acyclic objects. Similarly, 
the $K$-acyclic objects are exactly the $S$-local objects.

\bigskip
Let us now turn to the subject of chromatic homotopy theory.
A left localisation at a spectrum $E$ is said to be \textbf{finite}
if the class of $E$-acyclic objects is generated, in the sense of triangulated categories, by a set
of finite spectra. 

This is especially interesting in the case of the Johnson-Wilson theories $E(n)$. The Johnson-Wilson theories are Landweber exact modules over $BP$ with
\[
E(n)_* \cong \mathbb{Z}_{(p)}[v_1,...,v_n, v_n^{-1}], \,\,\, |v_i|=2p^i-2,\,\,\,p \,\,\mbox{prime}.
\]
Localisation with respect to $E(n)$ is smashing and is usually denoted by $L_n$ instead of $L_{E(n)}$. These localisations are of great importance to stable homotopy theory as they play a role in major structural results concerning the stable homotopy category such as the Nilpotency Theorem, Periodicity Theorem, Chromatic Convergence Theorem and Thick Subcategory Theorem. Further, $L_1$ equals localisation with respect to $p$-local complex topological $K$-theory whereas $L_2$ is related to elliptic homology theories.
One of the great open conjectures in stable homotopy theory, the \textbf{telescope conjecture} claims that localisation with respect to $E(n)$ is finite in the above sense. 

\begin{rmk}
This conjecture can be put into an even more concrete setting. Ravenel showed in \cite{Rav93} that the only finite localisations of spectra are of the form $L_{L_n^f \mathbb{S}}\Sp$ where $L_n^f \mathbb{S}$ is a finite localisation of the sphere. This is also a smashing localisation.
\end{rmk}

We can restate this in the language right localisations. By Lemma \ref{lem:smashing} we have that 
\[
L_n \Sp = L_\Gamma \Sp \,\,\,\mbox{for} \,\,\,\Gamma = \{ \Sigma^{k} \lambda: \mathbb{S}^k \longrightarrow L_n \mathbb{S}^k \,|\, k \in \mathbb{Z} \}.
\]
By Proposition \ref{prop:orthogonal} the question whether $L_n\Sp$ is finite is now equivalent to the question whether $R_{K_\Gamma}$ is finite. Hence we can now use the tools of right localisation to study the telescope conjecture in future research.

\end{document}